\documentclass[amstex,12pt,reqno]{amsart}
\usepackage{amsmath,amsfonts,amssymb,amsthm,enumerate,hyperref,multicol,graphics,float}
\usepackage[small]{caption}
\usepackage [dvips]{graphicx}
\newtheorem{lemma}{Lemma}

\newtheorem{theorem}{Theorem}

\newtheorem{example}{Example}

\newtheorem{remark}{Remark}
\begin{document}
\title{Approximation by $(p,q)$-Baskakov-Beta Operators}
\author{}
\maketitle

\begin{center}
\textbf{Neha Malik and Vijay Gupta}\\[0pt]
Department of Mathematics, Netaji Subhas Institute of Technology\\[0pt]
Sector 3 Dwarka, New Delhi-110078, India\\[0pt]
neha.malik\_nm@yahoo.com\\[0pt]
vijaygupta2001@hotmail.com
\vskip0.3in
\end{center}

\vskip0.3in
\textbf{Abstract.} In the present paper, we consider $(p,q)$-analogue of the Baskakov-Beta operators and using it, we estimate some direct results on approximation. Also, we represent the convergence of these operators graphically using MATLAB.

\section{Introduction}
Approximation theory has been an engaging field of research with abstract approximation to the core (cf. \cite{ms}). Varied operators with their approximation properties, mainly the quantitative one, have been discussed and studied by many researchers. It has been seen that the generalizations of several well known operators to quantum-calculus ($q$-calculus) were introduced in the last three decades and their approximation behavior were also discussed (see \cite{avr}, \cite{amc}, \cite{ZFVG} and \cite{vgrpa}). Further generalization of quantum variant is the post-quantum calculus, denoted by $(p,q)$-calculus. Very recently, some researchers studied in this direction (see \cite{TA}, \cite{inf} and \cite{VSSY}). Few basic definitions and notations mentioned below may be found in these papers and references therein.\\
 The $(p,q)$-numbers are given by
\begin{eqnarray*}
\left[ n\right] _{p,q}&:=& p^{n-1}+p^{n-2}q+p^{n-3}q^2+\cdots +pq^{n-2}+q^{n-1}\\
&=& \left\{ \begin{array}{cc}
         \frac{p^{n}-q^{n}}{p-q}  \ , & \ \ \mbox{if $p\ne q\ne 1$};\\
        n \ , & \ \ \mbox{if $p= q= 1$}.\end{array} \right.
\end{eqnarray*}
The $\left( p,q\right)$-factorial is given by $\left[ n\right] _{p,q}!=\prod\limits_{r=1}^{n}\left[ r\right] _{p,q}, \ \ n\geqslant 1, \ \ \left[
0\right] _{p,q}!=1.$
The $\left( p,q\right) $-binomial coefficient satisfies%
\begin{equation*}
\left[
\begin{array}{c}
n \\
r
\end{array}%
\right] _{p,q}=\frac{\left[ n\right] _{p,q}!}{\left[ n-r\right] _{p,q}!\left[
r\right] _{p,q}!}, \ \ 0\leqslant r\leqslant n.
\end{equation*}

Let $n$ be a non-negative integer, the $(p,q)$-Gamma
function is defined as
\begin{equation*}
\Gamma _{p,q}\left( n+1\right) =\frac{(p\ominus q)_{p,q}^{n}}{(p-q)^{n}}=%
\left[ n\right] _{p,q}!, \  \ 0<q<p,
\end{equation*}
where $(p\ominus q)_{p,q}^{n}=(p-q)(p^{2}-q^{2})(p^{3}-q^{3})\cdots (p^{n}-q^{n}).$
\\
The $(p,q)$-integral for $0<q<p \leqslant 1$ (generalized Jackson integral) is defined as
 \begin{eqnarray} \label{int}
 \int\limits_{0}^{a} f(x) \, d_{p,q}x = (p-q)a \sum\limits_{i=0}^{\infty} \frac{q^i}{p^{i+1}} \ f\left( \frac{a \, q^i}{p^{i+1}} \right), \ \ x\in [0,a].
 \end{eqnarray}
By simple computation, we get $$ \int\limits_{0}^{a} x^{n} \ d_{p,q}x = \frac{a^{n+1}}{[n+1]_{p,q}} \cdot$$
Also, the integral (\ref{int}) includes the nodes $x_{i} = x_{i}(p,q) = \frac{a \, q^i}{p^{i+1}}, \ \ i=0,1,\ldots, $
geometrically distributed in $(0,+ \infty),$ not only in $(0,a),$ as in the case $p=1$ (standard Jackson's $q$-integral). Moreover, one may observe that only a finite number of nodes in (\ref{int}) are outside $(0,a)$, i.e., those $x_i$ for which $q^i > p^{i+1}.$ Thus, the above definition of $(p,q)$-integral may be well utilized to define the $(p,q)$-extensions of well-known results.\\
For $m,n\in \mathbb{N},$ the $\left( p,q\right) $-Beta function of second kind considered in \cite{ag} is given by
\begin{eqnarray*}
B_{p,q}(m,n) = \int\limits_{0}^{\infty} \frac{t^{m-1}}{(1 \oplus pt)_{p,q}^{m+n}} \ d_{p,q}t,
\end{eqnarray*}
where the $(p,q)$-power basis is given by
\begin{eqnarray*}
(1 \oplus pt)_{p,q}^{m+n} &=& (1+pt)(p+pqt)(p^{2}+pq^{2}t)\cdots (p^{m+n-1}+pq^{m+n-1}t).
\end{eqnarray*}

Using the $(p,q)$-integration by parts: $$ \int\limits_{a}^{b} f(px) \ D_{p,q}g(x) \ d_{p,q}x = f(b) \, g(b) - f(a) \, g(a) - \int\limits_{a}^{b} g(qx) \ D_{p,q}f(x) \ d_{p,q}x ,$$
it was shown in \cite{ag} that the following relation is satisfied by the $(p,q)$-analogues of Beta and Gamma functions:
\begin{eqnarray*}
B_{p,q}(m,n) = \frac{q \ \ \Gamma _{p,q}(m) \, \Gamma _{p,q}(n)}{\left( p^{m+1} \ q^{m-1} \right)^{m/2} \ \Gamma _{p,q}(m+n)} \cdot \label{a4}
\end{eqnarray*}
As a special case, if $p=q=1,$ $B(m,n)=\Gamma(m) \ \Gamma(n) / \Gamma(m+n).$ It may be observed that in $(p,q)$-setting, order is important, which is the reason why $(p,q)$-variant of Beta function does not satisfy commutativity property, i.e., $B_{p,q}(m,n) \ne B_{p,q}(n,m).$ \\

For $n\in \mathbb{N},$ $x\in [0,\infty)$ and $0<q<p\leqslant 1,$ the $(p,q)$-analogue of Baskakov operators can be defined as
\begin{eqnarray*}
B_{n,p,q}\left( f,x\right) =\sum\limits_{k=0}^{\infty}b_{n,k}^{p,q}(x)f\left(\frac{%
p^{n-1}[k]_{p,q}}{q^{k-1}[n]_{p,q}}\right) ,
\end{eqnarray*}
where $(p,q)$-Baskakov basis function is given by
\begin{equation*}
b_{n,k}^{p,q}(x)=\left[
\begin{array}{c}
n+k-1 \\
k%
\end{array}%
\right] _{p,q}p^{k+n(n-1)/2} q^{k(k-1)/2}\frac{x^{k}}{(1\oplus x)_{p,q}^{n+k}%
} \cdot
\end{equation*}
Gupta \cite{inf} considered this form of $(p,q)$-Baskakov operators while studying its Kantorovich variant. This form was also considered by T. Acar et. al. in \cite{kant}.

\begin{remark} \label{r1}
It has been observed in \cite{inf} that the $(p,q)$-Baskakov operators satisfy the following recurrence relation:
\begin{eqnarray*}
\left[ n\right]
_{p,q}T_{n,m+1}^{p,q}(qx) &=& q \, p^{n-1} \ x \, (1+px) \ D_{p,q}[T_{n,m}^{p,q}(x)]
+\left[ n\right]
_{p,q} q \, x \, T_{n,m}^{p,q}(qx),
\end{eqnarray*}
where $T_{n,m}^{p,q}(x):=
B_{n,p,q}\left( e_m,x\right) = \sum\limits_{k=0}^{\infty}b_{n,k}^{p,q}(x)\left(\frac{%
p^{n-1}[k]_{p,q}}{q^{k-1}[n]_{p,q}}\right)^{m}.$ \\
Then, we have
\begin{eqnarray*}
B_{n,p,q}\left(e_{0},x\right)=1, \ \
B_{n,p,q}\left(e_{1},x\right)=x,
\end{eqnarray*}
\begin{eqnarray*}
B_{n,p,q}\left(e_{2},x\right)=\frac{[n+1]_{p,q} \ x^{2}+ p^{n-1} \, q \, x}{q \, [n]_{p,q}},
\end{eqnarray*}
where $e_{i}(t)=t^{i},$ $i=0,1,2.$
In case $p=1$, we get the $q$-Baskakov operators \cite{av-qBask}, \cite{ZFVG}. If $p=q=1,$ then these operators reduce to the well known Baskakov operators.
\end{remark}

\section{Construction of Operators and Moments}
In the year 1985, Sahai-Prasad \cite{sahai} introduced the Durrmeyer variant of the well known Baskakov operators. However, there were some technical problems in the main estimates of \cite{sahai}, which were later improved by Sinha et. al. \cite{sinha}. In this continuation, in 1994, Gupta proposed yet another Durrmeyer type generalization of Baskakov operators by taking the weights of Beta basis function. The operators discussed in \cite{ATTA} provide better approximation in simultaneous approximation than the usual Baskakov-Durrmeyer operators, studied in \cite{sinha}. This motivated us to study further in this direction and here, we propose the $(p,q)$-variant of Baskakov-Beta operators. \\

For $n\in \mathbb{N},$ $x\in [0,\infty)$ and $0<q<p\leqslant 1,$ the $(p,q)$-Baskakov-Beta operators are defined by:
\begin{eqnarray} \label{operatorBB}
D_{n}^{p,q}(f,x) = \sum\limits_{k=0}^{\infty} b_{n,k}^{p,q}(x) \ \frac{1}{B_{p,q}(k+1,n)} \int\limits_{0}^{\infty} \frac{t^{k}}{(1 \oplus pt)_{p,q}^{n+k+1}} \ f(q^{2}p^{n+k}t) \ d_{p,q}t ,
\end{eqnarray}
where $b_{n,k}^{p,q}(x) = \left[
\begin{array}{c}
n+k-1 \\
k
\end{array}
\right] _{p,q} p^{k+n(n-1)/2} \ q^{k(k-1)/2} \ \frac{x^{k}}{(1 \oplus x)_{p,q}^{n+k}} \cdot$

In the present article, we estimate the moments of these operators by using $(p,q)$-Beta functions and establish some direct results in terms of modulus of continuity of first and second order using K-functional. Finally, we provide weighted approximation estimate alongwith the rate of convergence.

\begin{lemma} \label{mom}
The following equalities hold:
\begin{enumerate}
  \item $D_{n}^{p,q}(1,x)=1;$
  \item $D_{n}^{p,q}(t,x)=\frac{[n]_{p,q} \, x + p^{n-2} \ q}{[n-1]_{p,q}},$ \ \mbox{for \ $n>1$};
  \item $D_{n}^{p,q}(t^{2},x) = \frac{x^{2} \, [n]_{p,q} \left( [n]_{p,q} + \frac{p^n}{q} \right)}{q \ [n-1]_{p,q} \ [n-2]_{p,q}} + \frac{x \, [n]_{p,q} \{p^{n-3} \ q^{2}+2 \, p^{n-2} \, q+p^{n-1}\} }{q \ [n-1]_{p,q} \ [n-2]_{p,q}} + \frac{p^{2n-5} \, q \, [2]_{p,q}}{[n-1]_{p,q} \ [n-2]_{p,q}},$ \ \mbox{for \ $n>2$}.
\end{enumerate}
\end{lemma}

\begin{proof} By Remark \ref{r1}, we have
\begin{eqnarray*}
D_{n}^{p,q}(1,x) &=& \sum\limits_{k=0}^{\infty} b_{n,k}^{p,q}(x) \ \frac{1}{B_{p,q}(k+1,n)} \int\limits_{0}^{\infty} \frac{t^{k}}{(1 \oplus pt)_{p,q}^{n+k+1}} \ d_{p,q}t\\
&=& \sum\limits_{k=0}^{\infty} b_{n,k}^{p,q}(x) \ \frac{1}{B_{p,q}(k+1,n)} \ B_{p,q}(k+1,n) \\
&=& \sum\limits_{k=0}^{\infty} b_{n,k}^{p,q}(x) \\
&=& B_{n,p,q}\left( 1,x\right) = 1.
\end{eqnarray*}
Next, using $[k+1]_{p,q} = q^{k} +p[k]_{p,q},$ we have
\begin{eqnarray*}
D_{n}^{p,q}(t,x) &=& \sum\limits_{k=0}^{\infty} b_{n,k}^{p,q}(x) \ \frac{1}{B_{p,q}(k+1,n)} \int\limits_{0}^{\infty} \frac{t^{k+1} \, q^{2} \, p^{n+k}}{(1 \oplus pt)_{p,q}^{n+k+1}} \ d_{p,q}t \\
&=& \sum\limits_{k=0}^{\infty} b_{n,k}^{p,q}(x) \ \frac{1}{B_{p,q}(k+1,n)} \ q^{2} \, p^{n+k} \, B_{p,q}(k+2,n-1)\\
&=& \sum\limits_{k=0}^{\infty} b_{n,k}^{p,q}(x) \ \frac{\Gamma_{p,q}(n+k+1)}{q^{[2-(k+1)k]/2} \, p^{-(k+1)(k+2)/2} \, \Gamma_{p,q}(k+1) \, \Gamma_{p,q}(n) } \ q^{2} \, p^{n+k} \\ && \times \ q^{[2-(k+2)(k+1)]/2} \, p^{-(k+2)(k+3)/2} \ \frac{\Gamma_{p,q}(k+2) \, \Gamma_{p,q}(n-1)}{\Gamma_{p,q}(n+k+1)}  \\
&=& \sum\limits_{k=0}^{\infty} b_{n,k}^{p,q}(x) \ p^{n-2} \ q^{1-k} \ \frac{[k+1]_{p,q}}{[n-1]_{p,q}} \\
&=& \frac{1}{[n-1]_{p,q}} \sum\limits_{k=0}^{\infty} b_{n,k}^{p,q}(x) \ p^{n-2} \ q^{1-k} \  (q^{k} +p[k]_{p,q}) \\
&=& \frac{p^{n-2} \ q}{[n-1]_{p,q}} \sum\limits_{k=0}^{\infty} b_{n,k}^{p,q}(x) + \frac{[n]_{p,q}}{[n-1]_{p,q}} \ \sum\limits_{k=0}^{\infty} b_{n,k}^{p,q}(x) \ \frac{p^{n-1} \, [k]_{p,q}}{q^{k-1} \, [n]_{p,q} } \cdot
\end{eqnarray*}

Using Remark \ref{r1}, we have
\begin{eqnarray*}
D_{n}^{p,q}(t,x) &=& \frac{p^{n-2} \ q}{[n-1]_{p,q}} \ B_{n,p,q}(1,x) + \frac{[n]_{p,q}}{[n-1]_{p,q}} \ B_{n,p,q}(t,x) \\
&=& \frac{p^{n-2} \ q}{[n-1]_{p,q}} \cdot 1 + \frac{[n]_{p,q}}{[n-1]_{p,q}} \cdot x\\
&=& \frac{[n]_{p,q} \, x + p^{n-2} \ q}{[n-1]_{p,q}} \cdot
\end{eqnarray*}

Further, using the identity $[k+2]_{p,q}=q^{k+1}+p \, q^{k}+p^{2} \, [k]_{p,q},$ we get
\begin{eqnarray*}
D_{n}^{p,q}(t^{2},x) &=& \sum\limits_{k=0}^{\infty} b_{n,k}^{p,q}(x) \ \frac{1}{B_{p,q}(k+1,n)} \int\limits_{0}^{\infty} \frac{t^{k+2} \, q^{4} \, p^{2(n+k)}}{(1 \oplus pt)_{p,q}^{n+k+1}} \ d_{p,q}t \\
&=& \sum\limits_{k=0}^{\infty} b_{n,k}^{p,q}(x) \ \frac{1}{B_{p,q}(k+1,n)} \ q^{4} \, p^{2(n+k)} \ B_{p,q}(k+3,n-2) \\
&=& \sum\limits_{k=0}^{\infty} b_{n,k}^{p,q}(x) \ \frac{\Gamma_{p,q}(n+k+1)}{q^{[2-(k+1)k]/2} \ p^{-(k+1)(k+2)/2} \ \Gamma_{p,q}(k+1) \ \Gamma_{p,q}(n)} \ q^{4} \, p^{2(n+k)} \\
 && \times \ q^{[2-(k+3)(k+2)]/2} \ p^{-(k+3)(k+4)/2} \ \frac{\Gamma_{p,q}(k+3) \ \Gamma_{p,q}(n-2)}{\Gamma_{p,q}(n+k+1)} \\
&=& \sum\limits_{k=0}^{\infty} b_{n,k}^{p,q}(x) \ p^{2n-5} \ q^{1-2k} \ \frac{[k+2]_{p,q} \ [k+1]_{p,q}}{[n-1]_{p,q} \ [n-2]_{p,q}} \\
&=& \frac{1}{[n-1]_{p,q} \ [n-2]_{p,q}} \ \sum\limits_{k=0}^{\infty} b_{n,k}^{p,q}(x) \ p^{2n-5} \ q^{1-2k} \ \left(q^{k+1}+p \, q^{k}+p^{2} \, [k]_{p,q}\right) \\
&& \times \ \left(q^{k}+p \, [k]_{p,q}\right) \\
&=& \frac{p^{2n-5} \, q^{2}+p^{2n-4} \, q}{[n-1]_{p,q} \ [n-2]_{p,q}} \ \sum\limits_{k=0}^{\infty} b_{n,k}^{p,q}(x) \\
&&+ \frac{p^{n-3} \, q \, [n]_{p,q} }{[n-1]_{p,q} \ [n-2]_{p,q}} \ \sum\limits_{k=0}^{\infty} b_{n,k}^{p,q}(x) \ \frac{p^{n-1} \, [k]_{p,q}}{q^{k-1} \, [n]_{p,q}} \\
&& + \frac{2 \, p^{n-2} \, [n]_{p,q} }{[n-1]_{p,q} \ [n-2]_{p,q}} \ \sum\limits_{k=0}^{\infty} b_{n,k}^{p,q}(x) \ \frac{p^{n-1} \, [k]_{p,q}}{q^{k-1} \, [n]_{p,q}} \\
 && + \frac{[n]_{p,q}^{2}}{q \ [n-1]_{p,q} \ [n-2]_{p,q}} \ \sum\limits_{k=0}^{\infty} b_{n,k}^{p,q}(x) \ \frac{p^{2n-2} \, [k]_{p,q}^{2}}{q^{2k-2} \, [n]_{p,q}^{2}}\cdot
\end{eqnarray*}
Again, using Remark \ref{r1}
\begin{eqnarray*}
D_{n}^{p,q}(t^{2},x) &=& \frac{p^{2n-5} \, q^{2}+p^{2n-4} \, q}{[n-1]_{p,q} \ [n-2]_{p,q}} \ B_{n,p,q}(1,x) \, + \, \frac{p^{n-3} \, q \, [n]_{p,q} }{[n-1]_{p,q} \ [n-2]_{p,q}} \ B_{n,p,q}(t,x) \\
 &&+ \, \frac{2 \, p^{n-2} \, [n]_{p,q} }{[n-1]_{p,q} \ [n-2]_{p,q}} \ B_{n,p,q}(t,x) \, + \, \frac{[n]_{p,q}^{2}}{q \ [n-1]_{p,q} \ [n-2]_{p,q}} \ B_{n,p,q}(t^{2},x) \\
&=& \frac{p^{2n-5} \, q^{2}+p^{2n-4} \, q}{[n-1]_{p,q} \ [n-2]_{p,q}} + \frac{x \, [n]_{p,q} \{ p^{n-3} \, q + 2 \, p^{n-2} \} }{[n-1]_{p,q} \ [n-2]_{p,q}} \\
&& + \left\{ x^{2} + \frac{p^{n-1} \, x}{[n]_{p,q}} \ \left( 1+ \frac{p \, x}{q} \right) \right\} \ \frac{[n]_{p,q}^{2}}{q \ [n-1]_{p,q} \ [n-2]_{p,q}} \\
&=& \frac{x^{2} \, [n]_{p,q} \left( [n]_{p,q} + \frac{p^n}{q} \right)}{q \ [n-1]_{p,q} \ [n-2]_{p,q}} + \frac{x \, [n]_{p,q} \{p^{n-3} \, q^{2}+2 \, p^{n-2} \, q+p^{n-1}\} }{q \ [n-1]_{p,q} \ [n-2]_{p,q}}\\
 && + \frac{p^{2n-5} \, q \, [2]_{p,q}}{[n-1]_{p,q}\ [n-2]_{p,q}} \cdot
\end{eqnarray*}
\end{proof}
\begin{remark} \label{cm}
Let $n>2$ and $x\in [0,\infty),$ then for $0<q<p\leqslant 1,$ we have the central moments as follows:
\begin{eqnarray*}
\mu_{n,1}^{p,q}(x) &:=& D_{n}^{p,q}((t-x),x) \\
&=& \frac{x \left( [n]_{p,q}-[n-1]_{p,q} \right)+ p^{n-2} \, q}{[n-1]_{p,q}}
\end{eqnarray*}
and
\begin{eqnarray*}
\mu_{n,2}^{p,q}(x) &:=& D_{n}^{p,q}((t-x)^{2},x) \\
   &=&  \frac{x^{2} \left\{ [n]_{p,q}\left( [n]_{p,q} + \frac{p^n}{q} \right) + q \, [n-1]_{p,q} \, [n-2]_{p,q} - 2 \, q \, [n]_{p,q} \, [n-2]_{p,q} \right\} }{q \, [n-1]_{p,q} \, [n-2]_{p,q}} \\
     &&+ \frac{x \{ [n]_{p,q} \left( p^{n-3} \, q^2 + 2 \, p^{n-2} \, q + p^{n-1} \right) - 2 \, p^{n-2} \, q^{2} \, [n-2]_{p,q} \} }{q \, [n-1]_{p,q} \, [n-2]_{p,q}} + \frac{[2]_{p,q} \, p^{2n-5} \, q}{[n-1]_{p,q} \, [n-2]_{p,q}} \cdot
\end{eqnarray*}
\end{remark}
\section{Direct Estimations}

In this section, we prove direct results using two different approaches, i.e., K-functional and Steklov mean. We also represent the convergence of the $(p,q)$-Baskakov-Beta operators using the software MATLAB.\\

We denote the norm $||f||=\sup\limits_{x\in[0,\infty)}|f(x)|$ on $C_B[0,\infty)$, the class of real valued continuous bounded functions. For $f\in C_B[0,\infty)$ and $\delta >0$, the $m$-th order modulus of continuity is defined as
$$\omega_m(f,\delta)=\sup\limits_{0\leqslant h\leqslant \delta} \ \sup\limits_{x\in[0,\infty)}|\Delta^m_hf(x)|,$$
where $\Delta_{h}$ is the forward difference and $\Delta_{h}^{m}= \Delta_{h} \left( \Delta_{h}^{m-1} \right)$ for $m\geqslant 1.$ In case $m=1$, we mean the usual modulus of continuity denoted by
$\omega(f,\delta).$ \\
The Peetre's $K$-functional is defined as
$$K_2(f,\delta)=\inf\limits_{g\in C_B^2[0,\infty)}\left\{||f-g||+\delta||g^{\prime\prime}||:g\in C_B^2[0,\infty)\right\},$$
where $$C_B^2[0,\infty)=\{g\in C_B[0,\infty): g^\prime, g^{\prime\prime}\in C_B[0,\infty)\}.$$
\begin{theorem}\label{t4} Let $f\in C_B[0,\infty)$ and $0<q<p\leqslant 1,$ then for every $x\in [0,\infty)$ and $n>2,$ the following inequality holds:
\begin{align*}
|D_{n}^{p,q}(f,x)-f(x)| &\leqslant \omega \left(f,|\mu_{n,1}^{p,q}(x)| \right) + C \ \omega_2\left(f,\sqrt{\mu_{n,2}^{p,q}(x) + \left( \mu_{n,1}^{p,q}(x)\right)^{2}} \, \right) , \\
\end{align*}
where $C$ is some positive constant.
\end{theorem}
\begin{proof}
Consider the following operator:
\begin{eqnarray} \label{check-operator}
\check{D}_{n}^{p,q}(f,x) = D_{n}^{p,q}(f,x) - f \left( \frac{[n]_{p,q} \, x + p^{n-2} \ q}{[n-1]_{p,q}} \right) + f(x), \ \ x\in [0,\infty)
\end{eqnarray}
Let $g\in C_B^2[0,\infty)$
and $x, \, t \in[0,\infty).$ By Taylor's expansion, we have
\begin{align*}
g(t) = g(x)+(t-x) \ g^\prime(x)+\int\limits_{x}^{t}(t-u) \ g^{\prime\prime}(u) \ du,
\end{align*}
Applying $\check{D}_{n}^{p,q},$ we get
\[ \check{D}_{n}^{p,q}(g,x)-g(x) = g^{\prime}(x) \, \check{D}_{n}^{p,q} \left((t-x),x \right) + \check{D}_{n}^{p,q} \left( \int\limits_{x}^{t}(t-u) \ g^{\prime\prime}(u) \ du ,x \right). \]
Hence,
\begin{align*}
& | \check{D}_{n}^{p,q}(g,x)-g(x)| \\
&\leqslant \bigg(D_{n}^{p,q} \bigg|\int\limits_{x}^{t}|t-u| \ |g^{\prime\prime}(u)| \ du\bigg|, \, x\bigg) + \bigg| \int\limits_{x}^{\frac{[n]_{p,q} \, x + p^{n-2} \ q}{[n-1]_{p,q}}}\left(\frac{[n]_{p,q} \, x + p^{n-2} \ q}{[n-1]_{p,q}}- u\right) g^{\prime\prime} (u) \ du \bigg| \\
&\leqslant D_{n}^{p,q}((t-x)^2, \, x) \ \|g^{\prime\prime}\| + \bigg| \int\limits_{x}^{\frac{[n]_{p,q} \, x + p^{n-2} \ q}{[n-1]_{p,q}}} \left(\frac{[n]_{p,q} \, x + p^{n-2} \ q}{[n-1]_{p,q}}- u\right) \ g^{\prime\prime}(u) \ du \bigg| \\
&\leqslant  \left\{ \mu_{n,2}^{p,q}(x) + \left(\frac{x \, \left( [n]_{p,q} - [n-1]_{p,q} \right)+ p^{n-2} \, q }{[n-1]_{p,q}}\right)^{2} \right\}  \  \|g^{\prime\prime}\| . \\
\end{align*}

Now, using operators (\ref{operatorBB}), we have, by Lemma \ref{mom},
\begin{eqnarray*}
|D_{n}^{p,q}(f,x)| &\leqslant& \sum\limits_{k=0}^{\infty} b_{n,k}^{p,q}(x) \ \frac{1}{B_{p,q}(k+1,n)} \int\limits_{0}^{\infty} \frac{t^{k}}{(1 \oplus pt)_{p,q}^{n+k+1}} \ |f(q^{2}p^{n+k}t)| \ d_{p,q}t.
\end{eqnarray*}
Hence, by (\ref{check-operator}), $ |D_{n}^{p,q}(f,x)| \leqslant 3 \, \|f\|.$ \\
Therefore
\begin{align*}
&|D_{n}^{p,q}(f,x)-f(x)| \\
 &\leqslant |\check{D}_{n}^{p,q}(f-g,x)-(f-g)(x)|+|\check{D}_{n}^{p,q}(g,x)-g(x)| + \bigg| f\left(\frac{[n]_{p,q} \, x + p^{n-2} \ q}{[n-1]_{p,q}}\right) - f(x) \bigg| \\
&\leqslant 4 \ \|f-g\|+ \left\{\mu_{n,2}^{p,q}(x) + \left( \frac{x \, \left( [n]_{p,q} - [n-1]_{p,q} \right) + p^{n-2} \, q }{[n-1]_{p,q}}\right)^{2} \right\}   \  \|g^{\prime\prime}\| \\
& \hskip0.2in + \omega \left( f, \frac{|x \, \left( [n]_{p,q} - [n-1]_{p,q} \right) + p^{n-2} \, q|}{[n-1]_{p,q}} \right).
\end{align*}
Lastly, taking infimum over all $g\in C_B^2[0,\infty)$, and using the inequality $K_2(f,\delta) \leqslant
C \, \omega_2(f,\sqrt{\delta})$, $\delta >0$ due to \cite{dl}, we get the desired assertion.
\end{proof}
\begin{example}
We show comparisons and some illustrative graphs for the convergence of $(p,q)$-analogue of Baskakov-Beta operators $D_{n}^{p,q}(f,x)$ for different values of the parameters $p$ and $q,$ such that $0< q < p \leqslant 1$.

For $x\in [0,\infty)$, $p = 0.9$ and $q = 0.8$, the convergence of the operators $D_{n}^{p,q}(f,x)$ to the function $f$, where $f(x)=18x^2-12x+2015$, for different values of $n$ is illustrated using MATLAB.
\begin{figure}[H]
\includegraphics[width=1.1\textwidth]{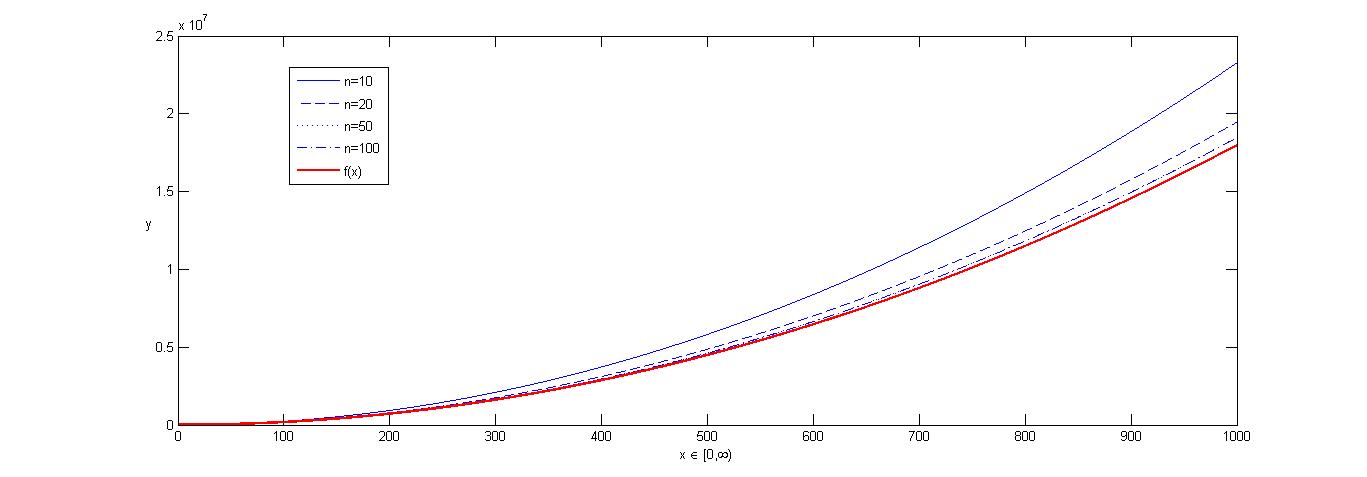}
\caption[]{$D_{n}^{0.9,0.8}(f,x)$ for $x\in[0,\infty)$, when $f(x)=18x^2-12x+2015.$}\label{Fig1}
\end{figure}
\end{example}
\begin{example}
For $x\in [0,\infty)$, $p = 0.9$ and $q = 0.75$, the convergence of the operators $D_{n}^{p,q}(f,x)$ to the function $f$, where $f(x)=25x^2-2x+7$, for different values of $n$ is illustrated using MATLAB.
\begin{figure}[H]
\includegraphics[width=1.1\textwidth]{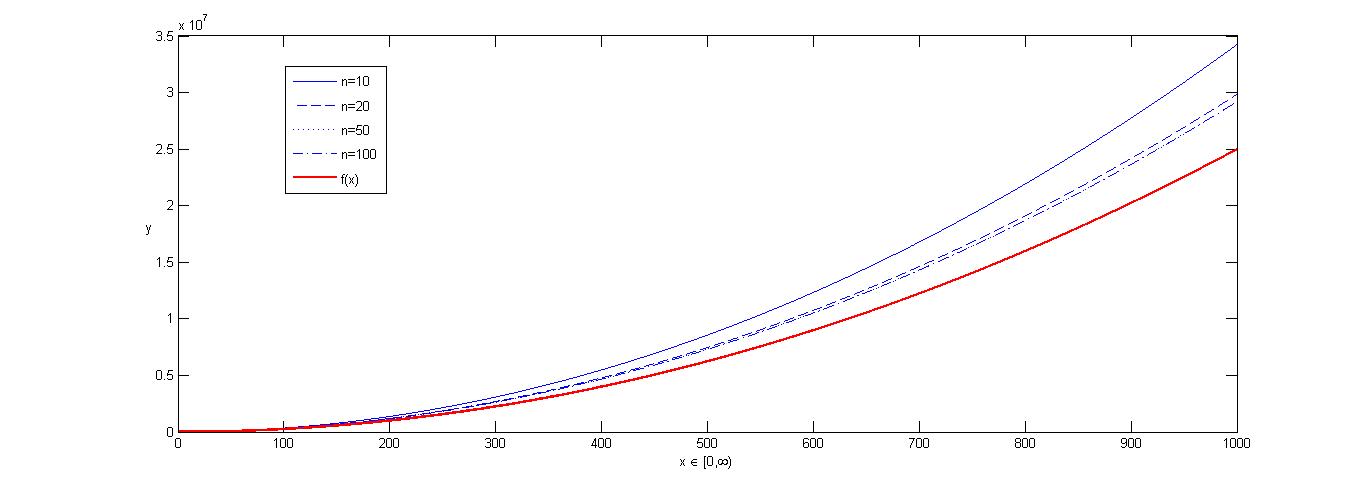}
\caption[]{$D_{n}^{0.9,0.75}(f,x)$ for $x\in[0,\infty)$, when $f(x)=25x^2-2x+7.$}\label{Fig3}
\end{figure}
\end{example}
Let $B_{\sigma}[0,\infty)$ be the space of all real valued functions on $[0,\infty)$ satisfying the condition $|f(x)|\leqslant C_f \ \sigma(x)$, where $C_f>0$ and $\sigma(x)$ is a weight function. Let $C_\sigma[0,\infty)$ be the space of all continuous functions in $B_\sigma[0,\infty)$ with the norm $\|f\|_\sigma=\sup\limits_{x\in[0,\infty)}\frac{|f(x)|}{\sigma(x)}$ and $C_\sigma^{0}[0,\infty) = \left\{f\in C_\sigma[0,\infty) : \lim\limits_{x\rightarrow \infty }\frac{|f(x)|}{\sigma(x)}<\infty\right\}$.
We consider $\sigma(x)= (1+x^{2})$ in the following two results. Also, we denote the modulus of continuity on $f$ on the closed interval $[0,\kappa],$ $\kappa >0$ by $$ \omega_{\kappa}(f,\delta) = \sup\limits{|t-x| \leqslant \delta} \ \sup\limits_{x,t \in [0, \kappa]} |f(t) - f(x)|.$$ It is easy to see that for $f \in C_\sigma[0,\infty),$ the modulus of continuity $\omega_{\kappa}(f,\delta)$ tends to zero.
Following is a theorem on the rate of convergence for the operators $D_{n}^{p,q}(f,x).$
\begin{theorem}
Let $f\in C_\sigma[0,\infty),$ $0<q<p \leqslant 1$ and $\omega_{\kappa +1}(f,\delta)$ be its modulus of continuity on the finite interval $[0, \kappa +1] \subset [0, \infty),$ where $\kappa > 0.$ Then, for every $n >2,$ $$ ||D_{n}^{p,q}(f) - f||_{C[0,\kappa]} \leqslant L \ \mu_{n,2}^{p,q}(x) + \omega_{\kappa +1} + 2 \, \omega \left( f, \sqrt{L \ \mu_{n,2}^{p,q}(x)} \, \right),$$ where $L=6 \ C_{f} \ (1+\kappa^{2}) (1+\kappa+\kappa^{2}).$
\end{theorem}

\begin{proof}
For $x\in [0, \kappa]$ and $t> \kappa +1,$ since $t-x >1,$ therefore we have
\begin{eqnarray} \label{et1}
|f(t)-f(x)| \leqslant C_{f} \ (2+ x^2 + t^2) \leqslant (2 + 3x^2 + 2(t - x)^2) \leqslant 6 \ C_f \ (1 + \kappa^2)(t - x)^2.
\end{eqnarray}
And for $x\in [0, \kappa]$ and $t \leqslant \kappa +1,$ we have
\begin{eqnarray} \label{et2}
|f(t) - f(x)| \leqslant \omega_{\kappa +1}(f, |t - x|) \leqslant \left(1 + \frac{|t -x|}{\delta} \right) \omega_{\kappa +1}(f, \delta),
\end{eqnarray}
with $\delta >0.$
Using above two equations (\ref{et1}) and (\ref{et2}),
\begin{eqnarray*}
|f(t) - f(x)| \leqslant 6 \ C_f \ (1 + \kappa^2) (t - x)^2 + \left(1 + \frac{|t -x|}{\delta} \right) \omega_{\kappa +1}(f, \delta),
\end{eqnarray*}
for $t \geqslant 0.$ \\
Now,
\begin{eqnarray*}
|D_{n}^{p,q}(f, x) - f(x)| &\leqslant& D_{n}^{p,q}(|f(t) - f(x)|, x) \\
&\leqslant& 6 \ C_f \ (1+ \kappa^2) D_{n}^{p,q}((t - x)^2, x) + \omega_{\kappa +1}(f, \delta) \ \left(1+\frac{1}{\delta} \, D_{n}^{p,q}((t - x)^2, x) \right)^{1/2}
\end{eqnarray*}
Using Remark \ref{cm} and Schwarz's inequality, for every $0<q<p \leqslant 1$ and $x\in [0, \kappa],$ we have
\begin{eqnarray*}
|D_{n}^{p,q}(f, x) - f(x)| &\leqslant& 6 \ C_f \ (1+ \kappa^2) \ \mu_{n,2}^{p,q}(x) + \omega_{\kappa +1}(f, \delta) \ \left(1+\frac{1}{\delta} \, \sqrt{\mu_{n,2}^{p,q}(x)} \right) \\
&\leqslant& L \ \mu_{n,2}^{p,q}(x)+ \omega_{\kappa +1}(f, \delta) \ \left(1+\frac{1}{\delta} \sqrt{\mu_{n,2}^{p,q}(x)} \right).
\end{eqnarray*}
Taking $\delta = \sqrt{L \ \mu_{n,2}^{p,q}(x)},$ the conclusion holds.
\end{proof}
Following is a direct estimate in weighted approximation.
\begin{theorem}
Let $p=p_{n},$ $q=q_{n}$ satisfying $0<q_{n}<p_{n}\leqslant 1$ and $p_{n} \to 1,$ $q_{n} \to 1,$ $p_{n}^{n} \to a,$ $q_{n}^{n} \to b$ as $n \to \infty.$ Then, for $f\in C_\sigma^{0}[0,\infty),$ we have
\[ \lim\limits_{n\rightarrow \infty}\|D_{n}^{p,q}\left(f\right)-f\|_\sigma = 0. \]
\end{theorem}
\begin{proof}
Due to the well-known Bohman-Korovkin theorem in \cite{korovkin}, it is sufficient to verify the following equation holds:
\begin{align*}
\lim_{n\rightarrow \infty}\|D_{n}^{p,q}\left(t^k,x\right)-x^k\|_\sigma = 0, \ \mbox{for} \ \text{k = 0, 1, 2.}
\end{align*}
By Lemma \ref{mom}, the result immediately follows for $k = 0.$ \\
Again, by Lemma \ref{mom}, we have:
\begin{eqnarray*}
\hskip-0.2in \|D_{n}^{p,q}(t,x)-x \|_\sigma &=& \sup_{x\in[0,\infty)} \bigg|\frac{[n]_{p,q} \, x + p^{n-2} \ q}{[n-1]_{p,q}}-x \bigg| \ \frac{1}{(1+x^2)}\\
&=& \sup_{x\in[0,\infty)} \left| \left\{\frac{x \left( [n]_{p,q}-[n-1]_{p,q} \right)+ p^{n-2} \, q}{[n-1]_{p,q}}\right\} \right| \ \frac{1}{(1+x^2)}
\end{eqnarray*}
and
\begin{eqnarray*}
\hskip-0.2in \|D_{n}^{p,q}(t^2,x)-x^2 \|_\sigma &=& \sup_{x\in[0,\infty)} \bigg| \frac{x^{2} \left\{ [n]_{p,q}\left( [n]_{p,q} + \frac{p^n}{q} \right) + q \, [n-1]_{p,q} \, [n-2]_{p,q} - 2 \, q \, [n]_{p,q} \, [n-2]_{p,q} \right\} }{q \, [n-1]_{p,q} \, [n-2]_{p,q}} \\
     && \hskip0.5in + \frac{x \left\{ [n]_{p,q} \left( p^{n-3} \, q^2 + 2 \, p^{n-2} \, q + p^{n-1} \right) - 2 \, p^{n-2} \, q^{2} \, [n-2]_{p,q} \right\} }{q \, [n-1]_{p,q} \, [n-2]_{p,q}} \\
     && \hskip0.5in+ \frac{[2]_{p,q} \, p^{2n-5} \, q}{[n-1]_{p,q} \, [n-2]_{p,q}}-x^2 \bigg| \ \frac{1}{(1+x^2)}\\
&=& \sup_{x\in[0,\infty)} \bigg| \frac{x^{2} \left\{ [n]_{p,q}\left( [n]_{p,q} + \frac{p^n}{q} \right) - 2 \, q \, [n]_{p,q} \, [n-2]_{p,q} \right\} }{q \, [n-1]_{p,q} \, [n-2]_{p,q}} \\
     && \hskip0.5in + \frac{x \left\{ [n]_{p,q} \left( p^{n-3} \, q^2 + 2 \, p^{n-2} \, q + p^{n-1} \right) - 2 \, p^{n-2} \, q^{2} \, [n-2]_{p,q} \right\} }{q \, [n-1]_{p,q} \, [n-2]_{p,q}} \\
     && \hskip0.5in+ \frac{[2]_{p,q} \, p^{2n-5} \, q}{[n-1]_{p,q} \, [n-2]_{p,q}} \bigg| \ \frac{1}{(1+x^2)},
\end{eqnarray*}
which shows that the result holds for $k=1$ and $k=2$ as well.\\
This completes the proof of the theorem.
\end{proof}
\begin{remark}
In order to obtain convergence estimates in the theorems discussed in the present article, we consider $0<q_{n} < p_{n} \leqslant 1,$ $p=p_{n}$ and $q=q_{n},$ such that $\lim\limits_{n \to \infty} p_{n} = 1 = \lim\limits_{n \to \infty} q_{n}$ and $\lim\limits_{n \to \infty} p_{n}^{n} = a,$ $\lim\limits_{n \to \infty} q_{n}^{n} = b.$ Thus, we have $\lim\limits_{n \to \infty} \frac{1}{[n]_{p_{n},q_{n}}} = 0.$
\end{remark}
\begin{remark}
The Srivastava-Gupta operators came into existence in the year 2003 (cf. \cite{GS}). Several further generalizations of these well-known operators were discussed (cf. \cite{ispir}, \cite{neha}, \cite{durvesh} and \cite{rani}). The $(p,q)$-variant of these operators cannot be defined for general form. Although, one can define the $(p,q)$-extensions for special cases of these operators separately. This leads to an open problem for our readers.
\end{remark}

{\bf Acknowledgement.} The authors are grateful to Prof. G. V. Milovanovi$\acute{c}$ for valuable discussions concerning the $(p,q)$-integral and their properties.


\end{document}